\documentclass[fleqn,reqno,11pt,a4paper,final]{amsart}

\usepackage[a4paper,left=30mm,right=30mm,top=30mm,bottom=30mm,marginpar=20mm]{geometry}
\usepackage{amsmath}
\usepackage{amssymb}
\usepackage{amsthm}
\usepackage{amscd}
\usepackage[ansinew]{inputenc}
\usepackage{cite}
\usepackage{bbm}
\usepackage{color}
\usepackage[english=american]{csquotes}
\usepackage[final]{graphicx}
\usepackage{hyperref}
\usepackage{calc}
\usepackage{mathptmx}

\graphicspath{{../Pictures/}}

\numberwithin{equation}{section}

\newtheoremstyle{thmlemcorr}{10pt}{10pt}{\itshape}{}{\bfseries}{.}{10pt}{{\thmname{#1}\thmnumber{ #2}\thmnote{ (#3)}}}
\newtheoremstyle{thmlemcorr*}{10pt}{10pt}{\itshape}{}{\bfseries}{.}\newline{{\thmname{#1}\thmnumber{ #2}\thmnote{ (#3)}}}
\newtheoremstyle{remexample}{10pt}{10pt}{}{}{\bfseries}{.}{10pt}{{\thmname{#1}\thmnumber{ #2}\thmnote{ (#3)}}}
\newtheoremstyle{ass}{10pt}{10pt}{}{}{\bfseries}{.}{10pt}{{\thmname{#1}\thmnumber{ A#2}\thmnote{ (#3)}}}

\theoremstyle{thmlemcorr}
\newtheorem{theorem}{Theorem}
\numberwithin{theorem}{section}

\newtheorem{corollary}[theorem]{Corollary}
\newtheorem{proposition}[theorem]{Proposition}

\theoremstyle{thmlemcorr*}
\newtheorem{theorem*}{Theorem}
\newtheorem{lemma*}[theorem]{Lemma}
\newtheorem{corollary*}[theorem]{Corollary}
\newtheorem{proposition*}[theorem]{Proposition}
\newtheorem{problem*}[theorem]{Problem}
\newtheorem{conjecture*}[theorem]{Conjecture}
\newtheorem{definition*}[theorem]{Definition}

\theoremstyle{remexample}
\newtheorem{remark}[theorem]{Remark}

\theoremstyle{ass}

\newcommand{\Nbb}{\mathbb{N}}

\newcommand{\Rbb}{\mathbb{R}}

\newcommand{\Tbb}{\mathbb{T}}
\newcommand{\T}{\mathbb{T}}

\DeclareMathOperator{\diverg}{div}
\DeclareMathOperator{\Div}{div}

\DeclareMathOperator{\dist}{dist}

\DeclareMathOperator{\supp}{supp}

\newcommand{\norm}[1]{\|#1\|}

\newcommand{\proofstep}[1]{\textit{#1}}

\newcommand{\vr}{\varrho}
\newcommand{\Grad}{\nabla}
\newcommand{\vu}{u}
\newcommand{\ep}{\varepsilon}
\newcommand{\DC}{C^\infty_c}
\newcommand{\intT}[1]{\int_{\T^d}  #1  \ \dx}

\newcommand{\dx}{ dx}
\newcommand{\dt}{dt}

\def\XXint#1#2#3{{\setbox0=\hbox{$#1{#2#3}{\int}$}
\vcenter{\hbox{$#2#3$}}\kern-.5\wd0}}

\renewcommand{\epsilon}{\varepsilon}
\renewcommand{\phi}{\varphi}

\begin{document}


\title[Energy Conservation for Compressible Euler]{Regularity and Energy Conservation for the Compressible Euler Equations}

\author{Eduard Feireisl}
\address{\textit{Eduard Feireisl:} Institute of Mathematics of the Academy of Sciences of the Czech Republic, \v{Z}itn\'a 25, CZ-115 67 Praha 1, Czech Republic }
\email{feireisl@math.cas.cz}

\author{Piotr Gwiazda}
\address{\textit{Piotr Gwiazda:}  Institute of Applied Mathematics and Mechanics, University of Warsaw, Banacha 2, 02-097 Warszawa, Poland}
\email{pgwiazda@mimuw.edu.pl}

\author{Agnieszka \'Swierczewska-Gwiazda}
\address{\textit{Agnieszka \'{S}wierczewska-Gwiazda:} Institute of Applied Mathematics and Mechanics, University of Warsaw, Banacha 2, 02-097 Warszawa, Poland }
\email{aswiercz@mimuw.edu.pl}

\author{Emil Wiedemann}
\address{\textit{Emil Wiedemann:} Hausdorff Center for Mathematics and Mathematical Institute, Universit\"{a}t Bonn, Endenicher Allee 60, 53115 Bonn, Germany}
\email{emil.wiedemann@hcm.uni-bonn.de}

\begin{abstract}
We give sufficient conditions on the regularity of solutions to the inhomogeneous incompressible Euler and the compressible isentropic Euler systems in order for the energy to be conserved. Our strategy relies on commutator estimates similar to those employed by P.\ Constantin et al.\ for the homogeneous incompressible Euler equations.
\end{abstract}







\maketitle




\section{Introduction}

We study in this paper the relationship between regularity and the conservation or dissipation of energy for two models of fluid dynamics: We consider the \emph{inhomogeneous incompressible Euler equations},
\begin{equation}\label{inhomintro}
\begin{aligned}
\partial_t(\rho u)+\diverg(\rho u\otimes u)+\nabla p&=0,\\
\partial_t\rho+\diverg(\rho u)&=0,\\
\diverg u&=0,
\end{aligned}
\end{equation}
as well as the \emph{compressible isentropic Euler equations},
\begin{equation}\label{compressibleintro}
\begin{aligned}
\partial_t(\rho u)+\diverg(\rho u\otimes u)+\nabla p(\rho)&=0,\\
\partial_t\rho+\diverg{\rho u}&=0.
\end{aligned}
\end{equation}
In both systems, $\rho\geq0$ is the scalar density of a fluid, $u$ is its velocity, and $p$ is the scalar pressure. Note however that in the incompressible case $p$ is an unknown, whereas in the compressible system it is a constitutively given function of the density. We consider these equations in any space dimension (it makes sense to study~\eqref{inhomintro} in two or more space dimensions and~\eqref{compressibleintro} in one or more space dimensions). While the compressible isentropic Euler equations are a well-accepted model for compressible flows, the inhomogeneous incompressible Euler equations have received somewhat less attention than its homogeneous special case (when $\rho\equiv1$ in~\eqref{inhomintro}). Nevertheless, a number of results on~\eqref{inhomintro} are available, among them~\cite{Ma1976, DaFa2011, Da2010, TrSh2016}.

Both systems have energies which are at least formally conserved. For~\eqref{inhomintro} the energy density is given by $\frac{1}{2}\rho|u|^2$ and for~\eqref{compressibleintro} it is $\frac{1}{2}\rho|u|^2+P(\rho)$, where $P$ is the \emph{pressure potential} associated with $p$. In compressible fluid dynamics, it is well-known that shocks may form, giving rise to energy dissipation. For incompressible fluids, the situation is more subtle, but it had been expected for a long time that the energy in fully turbulent flow should dissipate with a rate independent of viscosity and, accordingly, there should exist weak solutions of the incompressible Euler equations which do not conserve energy. Such weak solutions were eventually constructed by Scheffer~\cite{scheffer} and Shnirelman~\cite{shnirel}.

A common feature of these energy-dissipating solutions is that they necessarily exhibit a certain degree of irregularity. The question therefore is \emph{how much regularity is needed to guarantee the conservation of energy}. In the context of incompressible turbulence, this question is the subject of a famous conjecture of Onsager~\cite{On1949}, according to which energy should be conserved if the solution is H\"older continuous with exponent greater than $1/3$, while solutions with less regularity possibly dissipate energy. The first part of this assertion was proved (for the homogeneous incompressible Euler equations) in~\cite{constetiti, Ey1994, ChCoFrSh2008}, while significant progress has recently been made in constructing energy-dissipating solutions slightly below the Onsager regularity (the currently best available results are~\cite{BuDLIsSz2015, BuDLSz2015}).

We give in this paper sufficient conditions on the regularity of $\rho$ and $u$ to ensure the conservation of energy. Our approach relies on the idea of Constantin et al.~\cite{constetiti} to use suitable commutator estimates. Accordingly our regularity assumptions are stated in terms of Besov spaces $B_p^{\alpha,\infty}$ similarly to~\cite{constetiti}. However, since we have now two unknowns $\rho$ and $u$, it is possible to ``trade" regularity between the density and the velocity. In particular, if the velocity is sufficiently regular, then the energy will be conserved even if the density is only of bounded variation.

The term $\partial_t(\rho u)$ is nonlinear in $(\rho,u)$ and therefore requires a commutator estimate. In turn this makes it necessary to make an assumption on Besov regularity also \emph{in time} (for the homogeneous incompressible Euler system this is not needed). One may circumvent this time regularity assumption, as was done very recently in~\cite{TrSh2016} for the system~\eqref{inhomintro}, by formulating the equations in terms of the density and the momentum $m=\rho u$ and obtaining the energy conservation by multiplication of the momentum equation by $(\rho u)^\epsilon/\rho^\epsilon$ instead of $u^\epsilon$, where the index $\epsilon$ indicates a suitable regularization. Then however, $(\rho u)^\epsilon/\rho^\epsilon$ is no longer divergence-free, which requires a commutator estimate involving the pressure; as a consequence, a regularity assumption on the pressure has to be made. We choose to rather assume some time regularity, as this approach allows us to handle vacuum states (meaning $\rho=0$).

Theorems~\ref{inhomonsager} and~\ref{compressibleonsager} thus give energy conservation for~\eqref{inhomintro} and~\eqref{compressibleintro}, respectively, under the assumption of Besov regularity in time and space. Theorem~\ref{compressibleonsagerBV} states energy conservation for~\eqref{compressibleintro} with no assumption on regularity in time, but under the assumption that $\rho, u\in BV\cap C$ in space. Again, shocks provide an example that this result is optimal in the sense that the continuity assumption cannot be dropped. For Theorem~\ref{compressibleonsagerBV} we make use of a specific time regularization that allows us to deduce some time regularity from the space regularity; such an argument was already used in~\cite{FrMaRu2010} and later in~\cite{Wr2013, FeLiMa2015}.

Let us remark that statements similar to Theorems~\ref{inhomonsager} and~\ref{compressibleonsager} could also be proved for the \emph{Euler-Boussinesq equations} 
\begin{equation*}
\begin{aligned}
\partial_t u+\diverg(u\otimes u)+\nabla p&=\theta f,\\
\partial_t\theta+\diverg(\theta u)&=0,\\
\diverg u&=0
\end{aligned}
\end{equation*}
(without having to assume time regularity) and for Navier-Stokes systems (cf.~\cite{TrSh2016}).\\

\textbf{Acknowledgments.} This work was done while the authors were participating in the Research in Pairs Program at Mathematisches Forschungsinstitut Oberwolfach. They warmly thank the Institute for its kind hospitality and the excellent research environment it provided. 

P.G. and A.\'S.-G. received support from the National Science Centre
(Poland), 2015/18/M/ST1/00075. 

The work of E.F. has received funding from the European Research Council under the European Union's Seventh Framework Programme (FP7/2007-2013)/ ERC Grant Agreement 320078. The Institute of Mathematics of the Academy of Sciences of the Czech
Republic is supported by RVO:67985840.

\section{Besov Spaces}\label{besov}
In this section we briefly discuss some properties of the Besov space $B_p^{\alpha,\infty}(\Omega)$, where $\Omega=(0,T)\times\Tbb^d$ or $\Omega=\Tbb^d$. The said Besov space comprises those functions $w$ for which the norm
\begin{equation}\label{besovshift}
\norm{w}_{B_p^{\alpha,\infty}(\Omega)}:=\norm{w}_{L^p(\Omega)}+\sup_{\xi\in\Omega}\frac{\norm{w(\cdot+\xi)-w}_{L^p(\Omega\cap(\Omega-\xi))}}{|\xi|^\alpha}
\end{equation}
is finite (here $\Omega-\xi=\{x-\xi: x\in\Omega\}$).

Let $\eta\in C_c^\infty(\Rbb^N)$ for $N=1+d$ or $N=d$ (according to the choice of $\Omega$) be a standard mollifying kernel and set
\begin{equation*}
\eta^\epsilon(x)=\frac{1}{\epsilon^{N}}\eta\left(\frac{x}{\epsilon}\right).
\end{equation*}
With the notation $w^\epsilon=\eta^\epsilon*w$ for any function $w$, $w^\epsilon$ is well-defined on $\Omega^\epsilon=\{x\in\Omega: \dist(x,\partial\Omega)>\epsilon\}$.

It is then easy to check that the definition of the Besov spaces implies
\begin{equation}\label{besoveps}
\norm{w^\epsilon-w}_{L^p(\Omega)}\leq C\epsilon^\alpha\norm{w}_{B_p^{\alpha,\infty}(\Omega)}
\end{equation}
and
\begin{equation}\label{besovepsgradient}
\norm{\nabla w^\epsilon}_{L^p(\Omega)}\leq C\epsilon^{\alpha-1}\norm{w}_{B_p^{\alpha,\infty}(\Omega)}.
\end{equation}
Moreover, it is easy to see that $(B_p^{\alpha,\infty}\cap L^\infty)(\Omega)$ is an algebra, i.e.\ the product of two functions in this space is again contained in the space.

Let $BV((0,T)\times\Tbb^d)$ denote the space of functions of bounded variation.
\begin{proposition}\label{BVembedding}
$(BV\cap L^\infty)(\Omega)\subset B_p^{\frac{1}{p},\infty}(\Omega)$ for every $p\in[1,+\infty]$.
\end{proposition}
\begin{proof}
Let $w\in(BV\cap L^\infty)(\Omega)$. First observe that, trivially,
\begin{equation*}
\norm{w(\cdot+\xi)-w}_{L^\infty(\Omega\cap(\Omega-\xi))}\leq 2\norm{w}_{L^\infty}.
\end{equation*}
Next, let $(w_n)$ be a sequence of smooth functions that converge strictly to $w$ in $BV(\Omega)$. This means that
$w_n\to w$ in $L^1(\Omega)$, $\nabla w_n \to \nabla w$ weakly-(*) in $\mathcal{M}(\Omega)$, and, in addition, $TV(w_n)\to TV(w)$, where $TV$ denotes the total variation. For each $n$, $w_n$ is in $W^{1,1}(\Omega)$ and so we can estimate
\begin{equation*}
\norm{w_n(\cdot+\xi)-w_n}_{L^1(\Omega\cap(\Omega-\xi))}\leq |\xi|\norm{Dw_n}_{L^1}=|\xi|TV(w_n).
\end{equation*}
The left hand side converges to $\norm{w(\cdot+\xi)-w}_{L^1}$ because $w_n\to w$ in $L^1$, whereas the right hand side converges to $|\xi|TV(w)$ by the strict convergence.

Finally we interpolate between the $L^1$ and the $L^{\infty}$ estimate to obtain
\begin{equation*}
\norm{w(\cdot+\xi)-w}_{L^p(\Omega\cap(\Omega-\xi))}\leq C\norm{w}_{BV}^{1/p}\norm{w}_{L^\infty}^{1-1/p}|\xi|^{1/p}\leq C\norm{w}_{BV\cap L^\infty}|\xi|^{1/p}.
\end{equation*}
\end{proof}

\section{An Onsager-type Statement for the Inhomogeneous Incompressible Euler Equations}
Consider the inhomogeneous incompressible Euler system on $(0,T)\times\Tbb^d$:
\begin{equation}\label{inhom}
\begin{aligned}
\partial_t(\rho u)+\diverg(\rho u\otimes u)+\nabla p&=0,\\
\partial_t\rho+\diverg(\rho u)&=0,\\
\diverg u&=0.
\end{aligned}
\end{equation}

\begin{theorem}\label{inhomonsager}
Let $\rho$, $u$, $p$ be a solution of~\eqref{inhom} in the sense of distributions. Assume
\begin{equation}\label{besovhypo}
u\in B_p^{\alpha,\infty}((0,T)\times\Tbb^d),\hspace{0.3cm}\rho, \rho u\in B_q^{\beta,\infty}((0,T)\times\Tbb^d),\hspace{0.3cm}p\in L^{p^*}_{loc}((0,T)\times\Tbb^d)
\end{equation}
for some $1\leq p,q\leq\infty$ and $0\leq\alpha,\beta\leq1$ such that
\begin{equation}\label{exponenthypo}
\frac{2}{p}+\frac{1}{q}=1,\hspace{0.3cm}\frac{1}{p}+\frac{1}{p^*}=1,\hspace{0.3cm}2\alpha+\beta>1.
\end{equation}
Then the energy is locally conserved, i.e.
\begin{equation}\label{localenergy}
\partial_t\left(\frac{1}{2}\rho|u|^2\right)+\diverg\left[\left(\frac{1}{2}\rho|u|^2+p\right)u\right]=0
\end{equation}
in the sense of distributions on $(0,T)\times\Tbb^d$.
\end{theorem}

\begin{remark}\label{inhomremark}
\begin{enumerate}
\item If $\rho\equiv1$, then~\eqref{inhom} reduces to the homogeneous incompressible Euler equations, and the choice $p=q=3$ and $\alpha=\beta$ yields the classical result of Constantin et al.~\cite{constetiti} (except for the time regularity, which can be relaxed in this case).
\item It may seem unnatural to impose regularity requirements on $\rho$, $u$, and $\rho u$ at the same time. By the fact that the spaces $B_p^{\alpha,\infty}\cap L^\infty$ are algebras, however, we may replace the assumptions~\eqref{besovhypo} by
\begin{equation*}
u\in (B_p^{\alpha,\infty}\cap L^\infty)((0,T)\times\Tbb^d),\hspace{0.3cm}\rho\in (B_q^{\beta,\infty}\cap L^\infty)((0,T)\times\Tbb^d),\hspace{0.3cm}p\in L^{p^*}_{loc}((0,T)\times\Tbb^d)
\end{equation*}
with exponents that satisfy~\eqref{exponenthypo} and, in addition, $p\geq q$ and $\alpha\geq\beta$ (so that $B_p^{\alpha,\infty}\subset B_q^{\beta,\infty}$).
\item The hypothesis concerning integrability of the pressure may be completely omitted as soon as we are interested only in the total energy balance
\[
\frac{{\rm d}}{{\rm d}t} E(t) = 0 \ \mbox{in the sense of distributions in} \ (0,T),\ \mbox{where} \ E(t) = \frac{1}{2} \int_{\Tbb^d} \rho |u|^2\ \dx.
\]

\end{enumerate}
\end{remark}

\begin{proof}
We follow the strategy of~\cite{constetiti} and mollify the momentum equation in time and space (with a kernel and notation as in Section~\ref{besov}):
\begin{equation*}
\partial_t(\rho u)^\epsilon+\diverg(\rho u\otimes u)^\epsilon+\nabla p^\epsilon=0.
\end{equation*}
Let $\phi\in C_c^\infty((0,T)\times\Tbb^d)$ be a test function. Multiplication with $\phi u^\epsilon$ and integration in time and space gives
\begin{equation}\label{mollifiedequation}
\int_0^T\int_{\Tbb^d}\partial_t(\rho u)^\epsilon\cdot \phi u^\epsilon dxdt+\int_0^T\int_{\Tbb^d}\diverg(\rho u\otimes u)^\epsilon\cdot \phi u^\epsilon dxdt+\int_0^T\int_{\Tbb^d}\phi u^\epsilon\cdot\nabla p^\epsilon dxdt=0.
\end{equation}
Here we take $\epsilon>0$ small enough so that $\supp\phi\subset (\epsilon,T-\epsilon)\times\Tbb^d$. We can rewrite this equality, using appropriate commutators, as
\begin{equation}\label{smoothenergy}
\begin{aligned}
\int_0^T\int_{\Tbb^d}\partial_t(\rho^\epsilon u^\epsilon)\cdot \phi u^\epsilon& dxdt+\int_0^T\int_{\Tbb^d}\diverg((\rho u)^\epsilon\otimes u^\epsilon)\cdot \phi u^\epsilon dxdt\\
+&\int_0^T\int_{\Tbb^d}\phi u^\epsilon\cdot\nabla p^\epsilon dxdt=R_1^\epsilon+R_2^\epsilon,
\end{aligned}
\end{equation}
where
\begin{equation*}
R_1^\epsilon=\int_0^T\int_{\Tbb^d}\partial_t\left[\rho^\epsilon u^\epsilon-(\rho u)^\epsilon\right]\cdot \phi u^\epsilon dxdt
\end{equation*}
and
\begin{equation*}
R_2^\epsilon=\int_0^T\int_{\Tbb^d}\diverg\left[(\rho u)^\epsilon\otimes u^\epsilon-(\rho u\otimes u)^\epsilon\right]\cdot\phi u^\epsilon dxdt.
\end{equation*}
The first integral on the left hand side of~\eqref{smoothenergy} equals
\begin{equation}\label{integral1}
\int_0^T\int_{\Tbb^d}\phi\partial_t\rho^\epsilon |u^\epsilon|^2 dxdt + \frac{1}{2}\int_0^T\int_{\Tbb^d}\phi\rho^\epsilon\partial_t |u^\epsilon|^2 dxdt,
\end{equation}
whereas for the second integral in~\eqref{smoothenergy} we use the mollified version of the continuity equation,
\begin{equation*}
\partial_t\rho^\epsilon+\diverg(\rho u)^\epsilon=0,
\end{equation*}
to compute
\begin{equation}\label{integral2}
\begin{aligned}
\int_0^T\int_{\Tbb^d}\diverg((\rho u)^\epsilon\otimes u^\epsilon)\cdot\phi u^\epsilon dxdt&=-\int_0^T\int_{\Tbb^d}\phi ((\rho u)^\epsilon\otimes u^\epsilon): \nabla u^\epsilon dxdt\\
&\hspace{1cm}-\int_0^T\int_{\Tbb^d}u^\epsilon\cdot ((\rho u)^\epsilon\otimes u^\epsilon)\nabla\phi dxdt\\
&=-\frac{1}{2}\int_0^T\int_{\Tbb^d} \phi(\rho u)^\epsilon \cdot\nabla |u^\epsilon|^2 dxdt\\
&\hspace{1cm}-\int_0^T\int_{\Tbb^d}u^\epsilon\cdot ((\rho u)^\epsilon\otimes u^\epsilon)\nabla\phi dxdt\\
&=\frac{1}{2}\int_0^T\int_{\Tbb^d}\phi\diverg (\rho u)^\epsilon |u^\epsilon|^2 dxdt\\
&\hspace{1cm}+\int_0^T\int_{\Tbb^d}\frac{1}{2}\nabla\phi\cdot(\rho u)^\epsilon|u^\epsilon|^2-u^\epsilon\cdot ((\rho u)^\epsilon\otimes u^\epsilon)\nabla\phi dxdt\\
&=-\frac{1}{2}\int_0^T\int_{\Tbb^d}\phi\partial_t\rho^\epsilon |u^\epsilon|^2 dxdt\\
&\hspace{1cm}+\int_0^T\int_{\Tbb^d}\frac{1}{2}\nabla\phi\cdot(\rho u)^\epsilon|u^\epsilon|^2-(u^\epsilon\cdot (\rho u)^\epsilon) u^\epsilon\cdot\nabla\phi dxdt.\\
\end{aligned}
\end{equation}
The third integral in~\eqref{smoothenergy} can be handled using the divergence-free condition on $u$:
\begin{equation}\label{integral3}
\int_0^T\int_{\Tbb^d}\phi u^\epsilon\cdot\nabla p^\epsilon dxdt=-\int_0^T\int_{\Tbb^d}\nabla\phi\cdot u^\epsilon p^\epsilon dxdt.
\end{equation}

Thus, combining~\eqref{integral1},~\eqref{integral2},~\eqref{integral3}, and~\eqref{smoothenergy}, we find
\begin{equation*}
\begin{aligned}
\frac{1}{2}\int_0^T\int_{\Tbb^d}\partial_t\phi\rho^\epsilon|u^\epsilon|^2dxdt+\int_0^T\int_{\Tbb^d}&\nabla\phi\cdot\left[(u^\epsilon\cdot (\rho u)^\epsilon) u^\epsilon-\frac{1}{2}(\rho u)^\epsilon|u^\epsilon|^2+p^\epsilon u^\epsilon\right]dxdt\\
&=-R_1^\epsilon-R_2^\epsilon.
\end{aligned}
\end{equation*}
To prove our claim, it suffices to show $R_1^\epsilon, R_2^\epsilon\to0$ as $\epsilon\to0$.
Indeed, the fact that this is sufficient in order to prove the theorem follows from standard properties of mollifications together with our assumptions~\eqref{besovhypo}.

For $R_1^\epsilon$, we observe that
\begin{equation}\label{pointwisedecomp}
\begin{aligned}
\rho^\epsilon u^\epsilon-(\rho u)^\epsilon=&(\rho^\epsilon-\rho)(u^\epsilon-u)\\
&-\int_{-\epsilon}^\epsilon\int_{\Tbb^d}\eta^\epsilon(\tau,\xi)(\rho(t-\tau,x-\xi)-\rho(t,x))(u(t-\tau,x-\xi)-u(t,x))d\xi d\tau.
\end{aligned}
\end{equation}
The first part of $R_1^\epsilon$ therefore can be estimated by virtue of an integration by parts,~\eqref{besoveps},~\eqref{besovepsgradient}, and our assumptions as
\begin{equation*}
\begin{aligned}
&\left|\int_0^{T}\int_{\Tbb^d}\phi\partial_t[(\rho^\epsilon-\rho)(u^\epsilon-u)]\cdot u^\epsilon dxdt\right|\\
\leq\int_0^{T}\int_{\Tbb^d}&|\partial_t\phi(\rho^\epsilon-\rho)(u^\epsilon-u)\cdot u^\epsilon| dxdt+\int_0^{T}\int_{\Tbb^d}|\phi(\rho^\epsilon-\rho)(u^\epsilon-u)\cdot\partial_t u^\epsilon| dxdt\\
\leq & C\norm{\phi}_{C^1}\epsilon^\beta\epsilon^\alpha\norm{\rho}_{B_q^{\beta,\infty}}\norm{u}^2_{B_p^{\alpha,\infty}}+ C\norm{\phi}_{C^0}\epsilon^\beta\epsilon^\alpha\epsilon^{\alpha-1}\norm{\rho}_{B_q^{\beta,\infty}}\norm{u}^2_{B_p^{\alpha,\infty}}\to0
\end{aligned}
\end{equation*}
as $\epsilon\to0$.

For the second part of $\int\phi R_1^\epsilon dt$ according to~\eqref{pointwisedecomp}, we estimate (using integration by parts, Fubini,~\eqref{besovshift} and~\eqref{besovepsgradient})
\begin{equation*}
\begin{aligned}
&\left|\int_0^{T}\int_{\Tbb^d}\phi\partial_t\int_{-\epsilon}^\epsilon\int_{\Tbb^d}\eta^\epsilon(\tau,\xi)(\rho(t-\tau,x-\xi)-\rho(t,x))(u(t-\tau,x-\xi)-u(t,x))d\xi d\tau\cdot u^\epsilon dxdt\right|\\
&\leq C\norm{\phi}_{C^1}\epsilon^\beta\epsilon^\alpha\norm{\rho}_{B_q^{\beta,\infty}}\norm{u}^2_{B_p^{\alpha,\infty}}+C\norm{\phi}_{C^0}\epsilon^\beta\epsilon^\alpha\epsilon^{\alpha-1}\norm{\rho}_{B_q^{\beta,\infty}}\norm{u}^2_{B_p^{\alpha,\infty}}\to0
\end{aligned}
\end{equation*}
as $\epsilon\to0$.

The estimate for $R_2^\epsilon$ is very similar. We write
\begin{equation*}
\begin{aligned}
(\rho u)^\epsilon\otimes u^\epsilon&-(\rho u\otimes u)^\epsilon=((\rho u)^\epsilon-\rho u)\otimes(u^\epsilon-u)\\
&-\int_{-\epsilon}^\epsilon\int_{\Tbb^d}\eta^\epsilon(\tau,\xi)(\rho u(t-\tau,x-\xi)-\rho u(t,x))\otimes(u(t-\tau,x-\xi)-u(t,x))d\xi d\tau.
\end{aligned}
\end{equation*}
The first part of $R_2^\epsilon$ can be estimated similarly as for $R_1^\epsilon$:
\begin{equation*}
\begin{aligned}
&\left|\int_0^{T}\int_{\Tbb^d}\phi\diverg[((\rho u)^\epsilon-\rho u)\otimes(u^\epsilon-u)]\cdot u^\epsilon dxdt\right|\\
&\leq\norm{\phi}_{C^0}\int_0^{T}\int_{\Tbb^d}|((\rho u)^\epsilon-\rho u)\otimes(u^\epsilon-u):\nabla u^\epsilon|dxdt\\
&\hspace{1cm}+\norm{\phi}_{C^1}\int_0^{T}\int_{\Tbb^d}|(((\rho u)^\epsilon-\rho u)\otimes(u^\epsilon-u)) u^\epsilon|dxdt\\
&\leq C\norm{\phi}_{C^0}\epsilon^\beta\epsilon^\alpha\epsilon^{\alpha-1}\norm{\rho u}_{B_q^{\beta,\infty}}\norm{u}^2_{B_p^{\alpha,\infty}}+ C\norm{\phi}_{C^1}\epsilon^\beta\epsilon^\alpha\norm{\rho}_{B_q^{\beta,\infty}}\norm{u}^2_{B_p^{\alpha,\infty}}\to0.
\end{aligned}
\end{equation*}
Likewise, for the second part of  $R_2^\epsilon$ we get
\begin{equation*}
\begin{aligned}
&\left|\int_0^{T}\int_{\Tbb^d}\diverg\left\{\int_{-\epsilon}^\epsilon\int_{\Tbb^d}\eta^\epsilon(\tau,\xi)(\rho u(t-\tau,x-\xi)-\rho u(t,x))\otimes(u(t-\tau,x-\xi)-u(t,x))d\xi d\tau\right\}\cdot \phi u^\epsilon dxdt\right|\\
&\leq C\norm{\phi}_{C^0}\epsilon^\beta\epsilon^\alpha\epsilon^{\alpha-1}\norm{\rho u}_{B_q^{\beta,\infty}}\norm{u}^2_{B_p^{\alpha,\infty}}+C\norm{\phi}_{C^1}\epsilon^\beta\epsilon^\alpha\norm{\rho}_{B_q^{\beta,\infty}}\norm{u}^2_{B_p^{\alpha,\infty}}\to0,
\end{aligned}
\end{equation*}
which completes the proof.

\end{proof}

Although Theorem~\ref{inhomonsager} implies that the energy $E(t)=\frac{1}{2}\int_{\Tbb^d}\rho|u|^2dx$ is conserved in the sense of distributions, it is still conceivable that it takes different values on a set of times of measure zero. That this can in fact not occur under some further assumptions is the content of our next result:
\begin{corollary}\label{totalenergy}
If, in addition to the assumptions of Theorem~\eqref{inhomonsager},
\begin{equation*}
\rho, u\in L^\infty((0,T)\times\Tbb^d),
\end{equation*}
\begin{equation*}
\sup_{t\in[0,T]}(\norm{\rho}_{B_q^{\beta,\infty}(\Tbb^d)}+\norm{\rho u}_{B_q^{\beta,\infty}(\Tbb^d)})<\infty,
\end{equation*}
and $\beta>0$, then it follows that
\begin{equation}\label{weakcty}
\vr \in C_{{\rm weak}}( [0,T], L^2(\T^d)), \ \vr \vu \in C_{{\rm weak}}( [0,T], L^2 (\T^d, \Rbb^{d})),
\end{equation}
and, setting
\[
E(t) = \int_{ \rho(t, \cdot) > 0} \frac{ |\vr \vu|^2 }{\vr}(t) \ \dx,
\]
we have $E(s)=E(t)$ for all $s,t\in[0,T]$.
\end{corollary}
\begin{proof}
The weak continuity~\eqref{weakcty} follows in a standard way from the equations, cf. e.g.\ Appendix A in~\cite{euler2}. Moreover, by the Fr\'echet-Kolmogorov Theorem, $B_q^{\beta,\infty}(\Tbb^d)$ embeds compactly into $L^1(\Tbb^d)$. Therefore, if $t\in[0,T]$ and $t_n\to t$, then $\rho(t_n)\to\rho(t)$ and $\rho u(t_n)\to\rho u(t)$ strongly in $L^1$.

 Our aim is to prove strong continuity in $L^1$ of the energy density, i.e.\ as $n\to\infty$
\begin{equation}\label{strongL1}
1_{\rho(t_n) > 0} \frac{|\rho u|^2}{\rho}(t_n)\to 1_{\rho(t) > 0} \frac{|\rho u|^2}{\rho}(t) \quad in \  L^1(\T^d),
\end{equation}
Consider two sets: $F_\varepsilon=\{x\in \T^d: \varrho(t)\ge \varepsilon\}$ and $\T^d\setminus F_\ep$.

Step 1. We begin with  the set $F_\ep$. From the strong convergence in $L^1$ we conclude that, up to a subsequence (not relabeled), $(\rho(t_n))_{n\in\Nbb}$ and $(\rho u(t_n))_{n\in\Nbb}$ converge a.e.\ in $F_\ep$. By Egorov's theorem, for every $\delta>0$ there exists a set $E_\delta$  such that $|\T^d\setminus E_\delta|<\delta$, where the sequence $(\rho(t_n))_{n\in\Nbb}$ converges uniformly. This allows to conclude that on the set $E_\delta\cap F_\ep$ for sufficiently large $n$ also
$\rho(t_n)\ge \frac{\ep}{2}$. Since the function $(\rho,u)\mapsto \frac{|\rho u|^2}{\rho}$ is well defined and also continuous on $[\frac{\ep}{2},\infty)\times \T^d$, we deduce (as $\delta>0$ was arbitrary)
\begin{equation*}
\frac{|\rho u|^2}{\rho}(t_n)\to \frac{|\rho u|^2}{\rho}(t) \quad a.e. \ in\  F_\ep
\end{equation*}
as $n\to\infty$. Since $\frac{|\rho u|^2}{\rho}$ is uniformly bounded in $L^\infty((0,T)\times \T^d)$, the sequence $\left(\frac{|\rho u|^2}{\rho}(t_n)\right)_{n\in\Nbb}$ is equiintegrable in $L^1(F_\ep)$ and by Vitali's theorem we conclude strong convergence in $L^1(F_\ep)$.

Step 2. On the set $\T^d\setminus F_\ep$ observe that since $u$ is in $L^\infty((0,T)\times \T^d)$ and in the same way as in the previous step we conclude that $\rho(t_n)$ is sufficiently small for $n$ large enough, then $\varrho(t_n)|u|^2(t_n)$ converges to zero as $\ep\to0$. Thus \eqref{strongL1} holds.
\end{proof}

We end our discussion of the inhomogeneous incompressible Euler equations by recording a special case of Theorem~\ref{inhomonsager} which states that, if the velocity is sufficiently regular, we can ensure energy conservation even when the density has jump discontinuities. More precisely, we have:
\begin{corollary}\label{BV}
Let $\rho\in (BV\cap L^\infty)((0,T)\times\Tbb^d)$ and $u\in (B_3^{\alpha,\infty}\cap L^\infty)((0,T)\times\Tbb^d)$ be a solution of~\eqref{inhom}, where $\alpha>\frac{1}{3}$. Then the energy is conserved.
\begin{proof}
Combine Proposition~\ref{BVembedding}, Theorem~\ref{inhomonsager} and Remark~\ref{inhomremark}(2).
\end{proof}
\end{corollary}
In fact we can go to the extreme and assume only $\rho\in L^1$ and $u$ to be H\"older continuous with exponent greater than $1/2$, then still we have energy conservation.

\section{Energy Conservation for the Compressible Isentropic Euler Equations}

We consider now the isentropic Euler equations,
\begin{equation}\label{compressible}
\begin{aligned}
\partial_t(\rho u)+\diverg(\rho u\otimes u)+\nabla p(\rho)&=0,\\
\partial_t\rho+\diverg(\rho u)&=0.
\end{aligned}
\end{equation}
In contrast to the inhomogeneous incompressible system~\eqref{inhom}, the pressure $p=p(\rho)$ is no longer a Lagrange multiplier, but a constitutively given function of the density. We will make use of the so-called pressure potential defined by
\begin{equation*}
P(\rho)=\rho\int_1^\rho\frac{p(r)}{r^2}dr.
\end{equation*}
It turns out that we can prove a theorem on the compressible system that is similar to Theorem~\ref{inhomonsager}:
\begin{theorem}\label{compressibleonsager}
Let $\rho$, $u$ be a solution of~\eqref{compressible} in the sense of distributions. Assume 
\begin{equation*}
u\in B_3^{\alpha,\infty}((0,T)\times\Tbb^d),\hspace{0.3cm}\rho, \rho u\in B_3^{\beta,\infty}((0,T)\times\Tbb^d),\hspace{0.3cm}
0 \leq \underline{\rho} \leq \rho \leq \overline{\rho} \ \mbox{a.a. in} (0,T)\times\Tbb^d,
\end{equation*}
for some constants $\underline{\vr}$, $\overline{\vr}$, and
$0\leq\alpha,\beta\leq1$ such that
\begin{equation}\label{alphabeta}
\beta > \max \left\{ 1 - 2 \alpha; \frac{1 - \alpha}{2} \right\}.
\end{equation}
Assume further that $p \in C^2[\underline{\vr}, \overline{\vr}]$, and, in addition
\begin{equation}\label{pressure}
p'(0) = 0 \ \mbox{as soon as}\ \underline{\vr} = 0.
\end{equation}
Then the energy is locally conserved, i.e.
\begin{equation*}
\partial_t\left(\frac{1}{2}\rho|u|^2+P(\rho)\right)+\diverg\left[\left(\frac{1}{2}\rho|u|^2+p(\rho)+P(\rho)\right)u\right]=0
\end{equation*}
in the sense of distributions on $(0,T)\times\Tbb^d$.
\end{theorem}

\begin{remark}
\begin{enumerate}
\item For the isentropic pressure law $p(\rho)=\kappa \rho^\gamma$, $\gamma>1$, our $C^2$ assumption on $p$ is satisfied if either we exclude vacuum (i.e.\ we assume $\underline{\rho} > 0$) or we choose $\gamma\geq 2$.
\item The conservation of total energy follows, under appropriate additional assumptions, similarly to Corollary~\ref{totalenergy}.
\item The conclusion of Corollary~\ref{BV} remains true. This can be interpreted roughly as follows: Energy dissipating shocks can not form exclusively in the density, but only in the density and the velocity simultaneously.
\item Shocks also provide examples that show that our assumptions are sharp: A shock solution dissipates energy, but $\rho$ and $u$ are in $BV\cap L^\infty$, which embeds (see Proposition~\eqref{BVembedding}) into $B_3^{1/3,\infty}$. Hence such a solution satisfies~\eqref{alphabeta} with equality but fails to satisfy the conclusion. Besides, there are also (non-physical) $BV$ weak solutions that produce energy.
\item It is easy to check that under the hypotheses on the pressure $p$ stated above, we have $P \in C^2[\underline{\rho}, \overline{\rho}]$.
\item The hypothesis on temporal regularity can be relaxed provided
$\underline{\rho} > 0$, meaning there is no vacuum. Indeed, in this case $\frac{(\rho u)^\epsilon}{\rho^\epsilon}$ can be used as a test function in the momentum equation, cf.~\cite{TrSh2016}. 
\end{enumerate}
\end{remark}

We state another result on the compressible system, where we do not need to require Besov regularity in time:

\begin{theorem}\label{compressibleonsagerBV}
Assume that the pressure $p$ satisfies
\begin{equation} \label{pres2}
p \in C^2(0, \infty) \cap C[0, \infty), \ p(0) = 0.
\end{equation}
Let $\rho \in L^\infty((0,T) \times \T^d)$, $u \in L^\infty((0,T) \times \T^d)$ be a solution of~\eqref{compressible} in the sense of distributions. In addition, assume that
\[
u(t) \in BV \cap C(\T^d),\ \rho(t) \in BV \cap C(\T^d) \ \mbox{for a.a.}\ t \in [0,T]
\]
and
\begin{equation} \label{E1-thm}
u, \ \rho \in L^\infty(0,T; C(\T^d)), \ \nabla u, \ \nabla \rho \in L^\infty_{{\rm weak}-(*)} (0,T; \mathcal{M}(\T^d)).
\end{equation}
Then the energy is locally conserved, i.e.
\begin{equation*}
\partial_t\left(\frac{1}{2}\rho|u|^2+P(\rho)\right)+\diverg\left[\left(\frac{1}{2}\rho|u|^2+p(\rho)+P(\rho)\right)u\right]=0
\end{equation*}
in the sense of distributions on $(0,T)\times\Tbb^d$.
\end{theorem}

\begin{remark}
Following the proof of Corollary~\ref{totalenergy}, we observe that in the situation of Theorem~\ref{compressibleonsagerBV} the conservation of total energy holds for every time without any further assumptions.
\end{remark}

\subsection{Proof of Theorem~\ref{compressibleonsager}}
Just as in the previous section, we mollify the momentum equation, multiply by $\phi u^\epsilon$ for a test function $\phi$, and integrate in time and space to obtain~\eqref{mollifiedequation}. We rewrite this again using commutators to obtain
\begin{equation}\label{smoothenergy2}
\begin{aligned}
\int_0^T\int_{\Tbb^d}\partial_t(\rho^\epsilon u^\epsilon)\cdot \phi u^\epsilon& dxdt+\int_0^T\int_{\Tbb^d}\diverg((\rho u)^\epsilon\otimes u^\epsilon)\cdot \phi u^\epsilon dxdt\\
+&\int_0^T\int_{\Tbb^d}\phi u^\epsilon\cdot\nabla p(\rho^\epsilon) dxdt=R_1^\epsilon+R_2^\epsilon+R_3^\epsilon,
\end{aligned}
\end{equation}
the only difference to~\eqref{smoothenergy} being the pressure term, for which we introduce the commutator
\begin{equation*}
R_3^\epsilon=\int_0^T\int_{\Tbb^d}\nabla(p(\rho^\epsilon)-p(\rho)^\epsilon)\cdot\phi u^\epsilon dxdt,
\end{equation*}
while $R_1^\epsilon$ and $R_2^\epsilon$ are defined as in the last section.

For the first and second integral in~\eqref{smoothenergy2} we find, as before, the relations~\eqref{integral1} and~\eqref{integral2}, because the corresponding computations did not require $u^\epsilon$ to be divergence-free. It is only the third integral which requires additional attention. For this we first need to observe that, due to the chain rule and the mollified mass equation,
\begin{equation*}
\partial_tP(\rho^\epsilon)+\nabla P(\rho^\epsilon)\cdot u^\epsilon+P'(\rho^\epsilon)\rho^\epsilon\diverg u^\epsilon=S^\epsilon,
\end{equation*}
where we introduced the commutator
\begin{equation*}
S^\epsilon=P'(\rho^\epsilon)\diverg(\rho^\epsilon u^\epsilon-(\rho u)^\epsilon).
\end{equation*}
Observe also that, by definition of $P$,
\begin{equation*}
\rho^\epsilon P'(\rho^\epsilon)=P(\rho^\epsilon)+p(\rho^\epsilon).
\end{equation*}
After these preparations, we can compute the third integral in~\eqref{smoothenergy2} as
\begin{equation*}
\begin{aligned}
\int_0^T&\int_{\Tbb^d}\phi u^\epsilon\cdot\nabla p(\rho^\epsilon)dxdt\\
&=-\int_0^T\int_{\Tbb^d}\nabla\phi\cdot u^\epsilon p(\rho^\epsilon)dxdt-\int_0^T\int_{\Tbb^d}\phi p(\rho^\epsilon)\diverg u^\epsilon dxdt\\
&=-\int_0^T\int_{\Tbb^d}\nabla\phi\cdot u^\epsilon p(\rho^\epsilon)dxdt-\int_0^T\int_{\Tbb^d}\phi [\rho^\epsilon P'(\rho^\epsilon)-P(\rho^\epsilon)]\diverg u^\epsilon dxdt\\
&=-\int_0^T\int_{\Tbb^d}\nabla\phi\cdot u^\epsilon p(\rho^\epsilon)dxdt+\int_0^T\int_{\Tbb^d}\phi [\partial_t P(\rho^\epsilon)+\nabla P(\rho^\epsilon)\cdot u^\epsilon+P(\rho^\epsilon) \diverg u^\epsilon] dxdt\\
&\hspace{5cm}-\int_0^T\int_{\Tbb^d}\phi S^\epsilon dxdt\\
&=-\int_0^T\int_{\Tbb^d}\nabla\phi\cdot u^\epsilon p(\rho^\epsilon)dxdt-\int_0^T\int_{\Tbb^d}\partial_t\phi P(\rho^\epsilon)+\nabla\phi\cdot P(\rho^\epsilon) u^\epsilon dxdt\\
&\hspace{5cm}-\int_0^T\int_{\Tbb^d}\phi S^\epsilon dxdt\\
\end{aligned}
\end{equation*}
Putting everything together, we see that the theorem is proved once we have shown that $R_3^\epsilon$ and $\int\int \phi S^\epsilon dxdt$ tend to zero as $\epsilon\to0$. Indeed, the assumptions of Theorem~\ref{compressibleonsager} are stronger than those of Theorem~\ref{inhomonsager}, so that our previous estimates for $R_1^\epsilon$ and $R_2^\epsilon$ hold a fortiori.

Consider first $R_3^\epsilon$. Let us observe that if $p \in C^{2}([a,b])$ then
$$|p(s)-p(s_0)-p'(s_0)(s-s_0)|\le C (s-s_0)^2$$ for every $s,s_0\in [a,b]$. Note that the constant $C$ can be chosen independently of $s,s_0$.
Therefore
$$|p(\rho^\epsilon(t,x))-p(\rho(t,x))-p'(\rho(t,x))(\rho^\epsilon(t,x)-\rho(t,x))|\le C (\rho(t,x)-\rho^\epsilon (t,x))^2$$
and similarly

$$|p(\rho(t,y))-p(\rho(t,x))-p'(\rho(t,x))(\rho(t,y)-\rho(t,x))|\le C (\rho(t,x)-\rho(t,y))^2.$$

Applying convolution w.r.t.\ $y$ to the last inequality we get, after invoking Jensen's inequality:
$$|p(\rho)^\epsilon(t,x)-p(\rho(t,x))-p'(\rho(t,x))(\rho^\epsilon(t,x)-\rho(t,x))|\le C (\rho(t,x)-\rho(t,\cdot))^2 *_y \eta^\epsilon.$$

Therefore
$$|p(\rho^\epsilon(t,x))-p(\rho)^\epsilon(t,x)|\le C ((\rho(t,x)-\rho^\epsilon (t,x))^2 + (\rho(t,x)-\rho(t,\cdot))^2 *_y \eta^\epsilon.$$

We can thus estimate
\begin{equation*}
\begin{aligned}
|R_3^\epsilon|&=\left|\int_0^T\int_{\Tbb^d}\nabla(p(\rho^\epsilon)-p(\rho)^\epsilon)\cdot\phi u^\epsilon dxdt\right| \\
\leq& \int_0^T\int_{\Tbb^d}|\phi(p(\rho^\epsilon)-p(\rho)^\epsilon)\diverg u^\epsilon| dxdt+ \int_0^T\int_{\Tbb^d}|(p(\rho^\epsilon)-p(\rho)^\epsilon) u^\epsilon\cdot\nabla\phi| dxdt\\
&\leq C\norm{\phi}_{C^0}\epsilon^{2\beta}\epsilon^{\alpha-1}\norm{\rho}^2_{B_3^{\beta,\infty}}\norm{u}_{B_3^{\alpha,\infty}}+C\norm{\phi}_{C^1}\epsilon^{2\beta}\norm{\rho}^2_{B_3^{\beta,\infty}}\norm{u}_{B_3^{\alpha,\infty}}\to0.
\end{aligned}
\end{equation*}


Finally, let us estimate $\int\int\phi S^\epsilon dxdt$. We use~\eqref{pointwisedecomp} to split $S^\epsilon$ into two parts, so that we can estimate for the first part
\begin{equation*}
\begin{aligned}
&\left|\int_0^T\int_{\Tbb^d}\phi\diverg[(\rho^\epsilon-\rho)(u^\epsilon-u)]P'(\rho^\epsilon)dxdt\right|\\
&\leq\int_0^T\int_{\Tbb^d}|\nabla\phi(\rho^\epsilon-\rho)(u^\epsilon-u)P'(\rho^\epsilon)|dxdt+\int_0^T\int_{\Tbb^d}|\phi(\rho^\epsilon-\rho)(u^\epsilon-u)\cdot P''(\rho^\epsilon)\nabla\rho^\epsilon| dxdt\\
&\leq C\norm{\phi}_{C^1}\epsilon^\beta\epsilon^\alpha\norm{\rho}_{B_3^{\beta,\infty}}\norm{u}_{B_3^{\alpha,\infty}}+C\norm{\phi}_{C^0}\epsilon^\beta\epsilon^\alpha\epsilon^{\beta-1}\norm{\rho}^2_{B_3^{\beta,\infty}}\norm{u}_{B_3^{\alpha,\infty}}\to0
\end{aligned}
\end{equation*}
as $\epsilon\to0$. The second part is estimated similarly, thus completing the proof.

\qed

\subsection{Proof of Theorem~\ref{compressibleonsagerBV}}
\label{P}

%


\proofstep{Regularization}

Again we will regularize in space and time, but now with different parameters. Thus we consider the regularization by spatial convolution
\[
v^{\ep} (t, x)  = \int_{\T^d} \eta^\ep (x) v(t, x-y) \ {\rm d}y,
\]
where $\eta^\epsilon$ is as in Section~\ref{besov}, and,
following \cite{FrMaRu2010},
by time convolution
\[
v^h(t, x)= \int_0^T \chi_h (t - s) v(s, x) \ {\rm d}s;
\]
where
\[
\chi_h = \frac{1}{h} 1_{[-h,0]}.
\]
Note that $v^h$ enjoys lower time regularity than the standard regularization by smooth kernels. Specifically,
we identify
\begin{equation} \label{timeD}
\partial_t v^h = \frac{ v(t + h) - v(t)}{h} \in X
\end{equation}
provided $t$ and $t+h$ are Lebesgue points of a function $v \in L^1_{\rm loc} (0,T; X)$. In particular, the function $v^h$ is absolutely
continuous in $[h, T-h]$, with the derivative given by (\ref{timeD}) for a.a. $t \in (h, T-h)$.

We use the notation
\[
v^{\ep,h} = (v^\ep)^{h}.
\]
Regularizing  \eqref{compressible}$_1$ we get
\begin{equation} \label{R1}
\partial_t \vr^\ep  + \Div (\vr \vu)^\ep  = 0, \ \partial_t \vr^{\ep,h} + \Div (\vr \vu)^{\ep,h} = 0
\end{equation}
where the latter equation is satisfied in $(h,T-h)$.

Similarly we have
\begin{equation} \label{R2}
\partial_t (\vr \vu )^\ep  + \Div \left( (\vr \vu)^\ep \otimes \vu^\ep \right)  + \Grad (p(\vr))^\ep =
\Div \left( (\vr \vu)^\ep \otimes \vu^\ep \right) - \Div ( \vr \vu \otimes \vu )^\ep,
\end{equation}
and
\begin{equation} \label{R3}
\begin{split}
\partial_t \left( \vr^\ep \vu^\ep \right)^{h}  &+ \Div \left( (\vr \vu)^\ep \otimes \vu^\ep \right)^h  + \Grad (p(\vr))^{\ep,h} \\
&=
\left( \Div \left( (\vr \vu)^\ep \otimes \vu^\ep \right) - \Div \left(\vr \vu \otimes \vu \right)^{\ep} \right)^h
+ \partial_t \left(  \vr^\ep \vu^\ep  - (\vr \vu)^{\ep} \right)^h.
\end{split}
\end{equation}

\noindent
\proofstep{Total energy balance.}

We multiply (\ref{R3}) by $\varphi \vu^{\ep,h}$, where $\varphi \in \DC( (h,T-h) \times \T^d)$, and
integrate the resulting expression over $(0,T) \times \T^d$. Now, we proceed in several steps.

To handle the term containing the  time derivative we
use the identity (\ref{timeD})
to obtain
\[
\begin{split}
\int_0^T \intT{ \partial_t &\left( \vr^\ep \vu^\ep \right)^h\cdot \vu ^{\ep,h}  \varphi } \dt
=\int_\Rbb \intT{ \frac{\vr^\ep(t + h) \vu^\ep(t+h) -  \vr^\ep (t)  \vu^\ep (t) }{h} \cdot\vu^{\ep,h}  \varphi }  \dt \\
&=\int_\Rbb \intT{ \frac{\vr^\ep(t + h) \vu^\ep(t+h) -  \vr^\ep (t)  \vu^\ep (t) }{h} \cdot\vu^{\ep,h}  \varphi }  \dt
\\
&\quad- \int_\Rbb \intT{ \vr^\ep \frac{ \vu^\ep(t+h) -    \vu^\ep (t) }{h} \cdot\vu^{\ep,h}  \varphi } \dt
+ \frac{1}{2} \int_\Rbb \intT{ \vr^\ep \partial_t( |\vu^{\ep,h} |^2) \varphi } \dt
\\
&= \int_\Rbb \intT{ \frac{\vr^\ep(t + h)  -  \vr^\ep (t)   }{h} \vu^\ep(t+h) \cdot\vu^{\ep,h}  \varphi }  \dt
+\frac{1}{2} \int_\Rbb \intT{ \vr^\ep \partial_t (|\vu^{\ep,h} |^2) \varphi } \dt
\\
&= \int_\Rbb \intT{ \partial_t \vr^{\ep,h} \vu^\ep(t+h) \cdot\vu^{\ep,h}  \varphi }  \dt+\frac{1}{2} \int_\Rbb \intT{ \vr^\ep \partial_t (|\vu^{\ep,h} |^2) \varphi } \dt .
\end{split}
\]

Thus using the regularized equation of continuity
\[
\begin{split}
\int_0^T& \intT{ \partial_t \left( \vr^\ep \vu^\ep \right)^h \cdot\vu^{\ep,h}  \varphi }  \dt
\\
&=
- \int_\Rbb \intT{ \Div ( \vr \vu)^{\ep,h} \vu^\ep(t+h) \cdot\vu^{\ep,h}  \varphi }  \dt+\frac{1}{2} \int_\Rbb \intT{ \vr^\ep \partial_t ( |\vu^{\ep,h} |^2) \varphi } \dt
\\
&=
\int_\Rbb \intT{ ( \vr \vu)^{\ep,h} \cdot \Grad \left( \vu^\ep(t+h) \cdot\vu^{\ep,h}  \varphi \right) }  \dt+\frac{1}{2} \int_\Rbb \intT{ \vr^\ep \partial_t (|\vu^{\ep,h} |^2) \varphi } \dt.
\end{split}
\]

Consequently, we may infer that
\begin{equation} \label{T1}
\begin{split}
\int_0^T &\intT{ \partial_t \left( \vr^\ep \vu^\ep \right)^h \cdot\vu^{\ep,h}  \varphi }  \dt
\\
&= \frac{1}{2} \int_\Rbb \intT{ \vr^\ep \partial_t (|\vu^{\ep,h} |^2) \varphi } \dt + \frac{1}{2} \int_\Rbb \intT{ \partial_t \left( \vr^\ep \varphi \right)
|\vu^{\ep,h} |^2  } \dt
\\
&\quad+ \int_\Rbb \intT{ (\vr \vu)^{\ep,h} \cdot \Grad \left( \vu^\ep(t+h) \cdot\vu^{\ep,h}  \varphi \right) }  \dt
- \frac{1}{2} \int_\Rbb \intT{ \partial_t \left( \vr^\ep \varphi \right)
|\vu^{\ep,h} |^2  } \dt
\\
&= \frac{1}{2} \int_\Rbb \intT{ \partial_t \left( \vr^\ep  |\vu^{\ep,h} |^2 \varphi  \right) } \dt
+ \int_\Rbb \intT{ (\vr \vu)^{\ep,h} \cdot \Grad \left( \vu^\ep(t+h) \cdot\vu^{\ep,h}  \varphi \right) }  \dt\\
&\quad- \frac{1}{2} \int_\Rbb \intT{ (\vr \vu)^\ep \Grad \left( \varphi |\vu^{\ep,h} |^2 \right) }  \dt
- \frac{1}{2} \int_\Rbb \intT{ \vr^\ep
|\vu^{\ep,h} |^2 \partial_t \varphi  } \dt
\\
&=
- \frac{1}{2} \int_\Rbb \intT{ \vr^\ep
|\vu^{\ep,h} |^2 \partial_t \varphi  } \dt
\\
&\quad+ \int_\Rbb \intT{ ( \vr \vu)^{\ep,h} \cdot \Grad \left( \vu^\ep(t+h) \cdot\vu^{\ep,h}  \varphi \right) } \dt
- \frac{1}{2} \int_\Rbb \intT{ (\vr \vu)^\ep \Grad \left( \varphi |\vu^{\ep,h} |^2 \right) } \dt.
\end{split}
\end{equation}


For further computations observe that the convective term reads
\[
\begin{split}
\int_0^T \intT{ \Div ((\vr \vu)^\ep \otimes \vu^\ep )^h \cdot \vu^{\ep, h} \varphi } \ \dt
= - \int_\Rbb \intT{ ( (\vr \vu)^\ep \otimes \vu^\ep )^h : \Grad \left(  \vu^{\ep, h} \varphi \right) } \ \dt.
\end{split}
\]


In accordance with hypothesis (\ref{pres2}), we may write the pressure as
\begin{equation} \label{pres3}
p(\vr) = p_\delta(\vr) + (p(\vr) - p_\delta(\vr)), \ \mbox{with}\ p_\delta \in C^2[0, \infty), \ p_\delta(0) = p_\delta'(0) = 0,
\ | p - p_\delta| < \delta
\end{equation}
for any $\delta > 0$.
Accordingly, we get
\begin{equation}\label{T2}
\begin{split}
\int_0^T \intT{& (\Grad p_\delta(\vr))^{\ep, h}  \vu ^{\ep, h} \varphi } dt
\\
&= - \int_\Rbb \intT{ ( p_\delta(\vr) )^{\ep,h}  \Div \vu^{\ep,h} \varphi}  \dt -
 \intT{ ( p_\delta(\vr) )^{\ep,h}   \vu^{\ep,h} \cdot \Grad \varphi}  \dt + \mathcal{O}(\delta)\| \varphi \|_{C^1}
\\
&= - \int_\Rbb \intT{ p_\delta ( \vr^{\ep,h} )  \Div \vu^{\ep,h} \varphi}  \dt -
 \int_\Rbb \intT{ ( p_\delta (\vr) )^{\ep,h}  \vu^{\ep,h} \cdot \Grad \varphi}  \dt
\\
&\quad+ \int_\Rbb \intT{ \left( p_\delta ( \vr^{\ep,h} ) - ( p_\delta (\vr) )^{\ep,h}   \right) \Div \vu^{\ep,h} \varphi}  \dt + \mathcal{O}(\delta)\| \varphi \|_{C^1}.
\end{split}
\end{equation}

Now, we rewrite the equation of continuity \eqref{compressible}$_1$ as
\[
\partial_t \vr^{\ep,h} + \Div \left( \vr^{\ep,h} \vu^{\ep,h} \right) = -\Div (\vr \vu)^{\ep,h} + \Div \left( \vr^{\ep,h} \vu^{\ep,h} \right);
\]
whence, after renormalization, 
\begin{equation} \label{T3}
\partial_t P_\delta ( \vr^{\ep,h} ) + \Div \left( P_\delta(\vr^{\ep,h}) \vu^{\ep,h} \right)
+ p_\delta ( \vr^{\ep,h} )  \Div \vu^{\ep,h} = -P'_\delta(\vr^{\ep,h}) \left[ \Div (\vr \vu)^{\ep,h} - \Div \left( \vr^{\ep,h} \vu^{\ep,h} \right)
\right].
\end{equation}
where
\[
P_\delta (\vr) = \vr \int_1^\vr \frac{p_\delta(z)}{z^2} dz.
\]
\color{black}
As a consequence of (\ref{pres3}), we have $P_\delta \in C^2[0, \infty)$.

Thus going back to (\ref{T2}) we conclude that
\begin{equation} \label{T4}
\begin{split}
\int_0^T \intT{ (\Grad p(\vr))^{\ep, h} & \vu^{\ep, h} \varphi } =
\int_0^T \intT{ \left[ \partial_t P_\delta ( \vr^{\ep,h} ) + \Div \left( P_\delta(\vr^{\ep,h}) \vu^{\ep,h} \right) \right] \varphi } \dt\\
&-
 \intT{ (p_\delta(\vr) )^{\ep,h}   \vu^{\ep,h} \cdot \Grad \varphi}  \dt
 \\
& + \int_\Rbb \intT{ \left( p_\delta ( \vr^{\ep,h} ) - (p_\delta(\vr))^{\ep,h}   \right) \Div \vu^{\ep,h} \varphi}  \dt\\
 &+ \int_\Rbb \intT{ P'_\delta(\vr^{\ep,h}) \left[ \Div (\vr \vu)^{\ep,h} - \Div \left( \vr^{\ep,h} \vu^{\ep,h} \right) \varphi \right] } \dt
 + \mathcal{O}(\delta) \| \varphi \|_{C^1}
\\
&=
- \int_\Rbb \intT{ P ( \vr^{\ep,h} ) \partial_t \varphi }  \dt - \int_0^T \intT{ P(\vr^{\ep,h}) \vu^{\ep,h} \cdot \Grad \varphi } \dt
\\
&-
 \intT{ ( p(\vr))^{\ep,h}   \vu^{\ep,h} \cdot \Grad \varphi}  \dt
\\
&  + \int_\Rbb \intT{ \left( p_\delta ( \vr^{\ep,h} ) - ( p_\delta(\vr) )^{\ep,h}   \right) \Div \vu^{\ep,h} \varphi}  \dt\\
& + \int_\Rbb \intT{ P'_\delta(\vr^{\ep,h}) \left[ \Div (\vr \vu)^{\ep,h} - \Div \left( \vr^{\ep,h} \vu^{\ep,h} \right) \varphi \right] } \dt
+ \mathcal{O}(\delta) \| \varphi \|_{C^1}.
\end{split}
\end{equation}
\color{black}

Thus summing up (\ref{T1}), (\ref{T2}) and (\ref{T4}) we obtain 
\[
\begin{split}
- \frac{1}{2} \int_\Rbb& \intT{ \vr^\ep
|\vu^{\ep,h} |^2 \partial_t \varphi  } \dt
\\
&+ \int_\Rbb \intT{ (\vr \vu)^{\ep,h} \cdot \Grad \left( \vu^\ep(t+h) \cdot\vu^{\ep,h}  \varphi \right) }  \dt
- \frac{1}{2} \int_\Rbb \intT{ (\vr \vu)^\ep \Grad \left( \varphi |\vu^{\ep,h} |^2 \right) } \dt
\\
&- \int_\Rbb \intT{ ( (\vr \vu)^\ep \otimes \vu^\ep )^h : \Grad \left( \vu^{\ep, h} \varphi \right) }  \dt
\\
&- \int_\Rbb \intT{ P ( \vr^{\ep,h} ) \partial_t \varphi }  \dt - \int_\Rbb \intT{ P(\vr^{\ep,h}) \vu^{\ep,h} \cdot \Grad \varphi } \dt
\\
&-
\int_\Rbb \intT{  (p(\vr))^{\ep,h}   \vu^{\ep,h} \cdot \Grad \varphi}  \dt
\\
&  + \int_\Rbb \intT{ \left( p_\delta ( \vr^{\ep,h} ) - ( p_\delta (\vr) )^{\ep,h}   \right) \Div \vu^{\ep,h} \varphi}  \dt\\
& + \int_\Rbb \intT{ P'_\delta(\vr^{\ep,h}) \left[ \Div (\vr \vu)^{\ep,h} - \Div \left( \vr^{\ep,h} \vu^{\ep,h} \right)  \right] \varphi }  \dt
 \\
& - \int_\Rbb \intT{ \Big( \left( \Div \left( (\vr \vu)^\ep \otimes \vu^\ep \right) - \Div \left( \vr \vu \otimes \vu \right)^{\ep} \right)^h
+ \partial_t \left(  \vr^\ep \vu^\ep  - (\vr \vu)^{\ep} \right)^h \Big) \cdot \vu^{\ep,h} \varphi } \dt= \mathcal{O}(\delta) \| \varphi \|_{C^1},
\end{split}
\]
\color{black}
which may be rewritten as
\begin{equation} \label{T5}
\begin{split}
\int_\Rbb& \intT{ \left( \frac{1}{2} \vr^\ep
\left|\vu^{\ep,h} \right|^2 +  P ( \vr^{\ep,h} )  \right) \partial_t \varphi     }  \dt  \\
&+ \int_\Rbb \intT{ \left( ( (\vr \vu)^\ep \otimes \vu^\ep )^h \cdot  \vu^{\ep, h} + \frac{1}{2} (\vr \vu)^\ep |\vu^{\ep,h} |^2 - (\vr \vu)^{\ep,h} \left( \vu^\ep(t+h) \cdot\vu^{\ep,h} \right) \right) \cdot \Grad \varphi }  \dt
\\
&+ \int_\Rbb \intT{ \Big( P(\vr^{\ep,h}) +  (p(\vr) )^{\ep,h} \Big)  \vu^{\ep,h} \cdot \Grad \varphi} \dt = \sum_{i=1}^{5} E^1_{\ep,h} + \mathcal{O}(\delta) \| \varphi \|_{C^1}
\end{split}
\end{equation}
where 
\[
\begin{split}
E^1_{\ep,h} =&  \int_\Rbb \intT{ (\vr \vu)^{\ep,h} \cdot \Grad \left( \vu^\ep(t+h) \cdot\vu^{\ep,h} \right)  \varphi  }  \dt
- \frac{1}{2} \int_\Rbb \intT{ (\vr \vu)^\ep \cdot \Grad \left( |\vu^{\ep,h} |^2 \right) \varphi } \dt
\\
&- \int_\Rbb \intT{ ( (\vr \vu)^\ep \otimes \vu^\ep )^h : \Grad \left(  \vu^{\ep, h}  \right) \varphi }  \dt
\end{split}
\]
\[
E^2_{\ep,h} = \int_\Rbb \intT{ \left( p_\delta ( \vr^{\ep,h} ) - ( p_\delta (\vr) )^{\ep,h}   \right) \Div \vu^{\ep,h} \varphi} \dt
\]
\[
E^3_{\ep,h} =  \int_\Rbb \intT{ P'_\delta(\vr^{\ep,h}) \left[ \Div (\vr \vu)^{\ep,h} - \Div \left( \vr^{\ep,h} \vu^{\ep,h} \right) \right] \varphi } \dt
\]
\[
E^4_{\ep,h} = -\int_\Rbb \intT{ \Big( \left( \Div \left( (\vr \vu)^\ep \otimes \vu^\ep \right) - \Div \left( \vr \vu \otimes \vu \right)^{\ep} \right)^h
\Big) \cdot \vu^{\ep,h} \varphi }  \dt
\]
and
\[
E^5_{\ep,h} =- \int_\Rbb \intT{ \Big( \partial_t \left(  \vr^\ep \vu^\ep  - (\vr \vu)^{\ep} \right)^h \Big) \cdot \vu^{\ep,h} \varphi }\dt.
\]
\color{black}

\noindent
\proofstep{Estimating the errors}
\label{E}

We perform the limit \emph{first} $\ep \to 0$ and \emph{second} $h \to 0$. We start with the last integral $E_5$ rewriting it as
\[
E^5_{\ep,h} =  \int_\Rbb \intT{ \Big( \left( \vr^\ep \vu^\ep  - (\vr \vu)^{\ep} \right)^h \Big) \cdot \partial_t (\vu^{\ep,h} \varphi) } \dt.
\]
For $h$ fixed, the term $\partial_t (\vu^{\ep,h} \varphi)$ remains uniformly bounded as $\vu$ belongs to $L^\infty$. On the other hand,
by the same token
\[
\vr^\ep \vu^\ep  - (\vr \vu)^{\ep} \to 0 \ \mbox{in}\ L^p((0,T) \times \T^d) \ \mbox{for any}\ 1 \leq p < \infty;
\]
whence $E^5_{\ep,h}$ vanishes for $\ep \to 0$.

The next step is to rewrite $E^1_{\ep,h}$ as
\[
\begin{split}
E^1_{\ep,h} =  &\int_\Rbb \intT{ (\vr \vu)^{\ep,h} (t-h) \cdot \Grad  \vu^{\ep} \cdot\vu^{\ep,h}(t-h)   \varphi(t-h)  }  \dt
\\
&+ \int_\Rbb \intT{ (\vr \vu)^{\ep,h} \cdot  \vu^\ep(t+h) \cdot  \Grad\vu^{\ep,h}   \varphi  }  \dt
\\
&- \int_\Rbb \intT{ (\vr \vu)^\ep \cdot \Grad \vu^{\ep,h} \cdot \vu^{\ep,h}  \varphi } \dt
\\
&- \int_\Rbb \intT{ ( (\vr \vu)^\ep \otimes \vu^\ep )^h : \Grad   \vu^{\ep, h}   \varphi }  \dt
\end{split}
\]

At this stage we invoke the assumption that 
\begin{equation*}
{\rm ess} \sup_{t \in [0,T]} (\| \vr (t, \cdot) \|_{BV \cap C(\T^d)} + \| \vu (t, \cdot) \|_{BV \cap C(\T^d; \Rbb^N)})<\infty.
\end{equation*}
\color{black}
Consequently,
\begin{align*}
(\vr \vu)^{\ep, h}& \to (\vr \vu)^h &&\mbox{in}\ \ C([h,T-h] \times \T^d ; \Rbb^d), \\
 \vu^{\ep, h} &\to \vu^h &&\mbox{in}\  \ C([h,T-h] \times \T^d ; \Rbb^d)
\ \mbox{as}\ \ep \to 0,\\
\Grad \vu^\ep &\to \Grad \vu &&\mbox{in}\ \ L^p_{\rm weak-(*)} (0,T; \mathcal{M}(\T^d; \Rbb^{d \times d})) \ \mbox{as}\ \ep \to 0, \\
\Grad \vu^{\ep,h} &\to \Grad \vu^h &&\mbox{in}\ \ C_{\rm weak-(*)} ([h,T-h]; \mathcal{M}(\T^d; \Rbb^{d \times d})) \ \mbox{as}\ \ep \to 0,\\
(\vr \vu)^\ep &\to \vr \vu &&\mbox{in}\ \ L^p(0,T; C(\T^d; \Rbb^d)), \\
 \vu^\ep &\to \vu &&\mbox{in}\ \ L^p(0,T; C(\T^d; \Rbb^d))
\ \mbox{for any}\ 1 \leq p < \infty \ \mbox{as}\ \ep \to 0.
\end{align*}
\color{black}
Thus we may conclude that
\[
\begin{split}
E^1_{\ep,h} \to E^1_h = & \int_\Rbb \intT{ ( \vr \vu)^h (t-h) \cdot \Grad  \vu \cdot \vu^h(t-h)   \varphi(t-h)  }  \dt
\\
&+ \int_\Rbb \intT{ ( \vr \vu)^h \cdot  \vu(t + h) \cdot  \Grad \vu ^h   \varphi  }  \dt
\\
&- \int_\Rbb \intT{(\vr \vu) \cdot \Grad  \vu^h \cdot \vu^h  \varphi }  \dt
\\
&- \int_\Rbb \intT{ ( \vr \vu \otimes \vu )^h : \Grad  \vu^h   \varphi }  \dt
\end{split}
\]
as $\ep \to 0$.

Now, observe that
\[
v^h \to v \ \mbox{in}\ L^p(0,T; X) \ \mbox{for}\ v \in L^\infty(0,T; X) \ \mbox{and any}\ 1 \leq p < \infty,
\]
and
\[
v(\cdot +h) - v \to 0 \ \mbox{in} \ L^p(0,T; X) \ \mbox{as}\ h \to 0 \ \mbox{for} \ v \in L^\infty(0,T; X)\ \mbox{and any}\ 1 \leq p < \infty,
\]
where $X$ is a Banach space. Thus we may infer that that $E^1_h \to 0$ as $h \to 0$.

It is easy to observe that the limits in $E^2_{\ep,h}$, $E^4_{\ep,h}$ can be performed in the same way.

Finally, we write $E^3_{\ep,h}$ as 
\[
E^3_{\ep,h} =- \int_\Rbb \intT{ \Grad (P'_\delta (\vr^{\ep,h}) \varphi) \left[ (\vr \vu)^{\ep,h} - \left( \vr^{\ep,h} \vu^{\ep,h} \right) \right]  }  \dt,
\]
where
\[
\Grad P'_\delta(\vr^{\ep,h} )= P''_\delta(\vr^{\ep,h}) \Grad \vr^{\ep,h},
\]
\color{black}
and apply the same arguments to conclude. As $\delta > 0$ in (\ref{T5}) can be taken arbitrarily small, the desired conclusion follows.
\qed

\end{document}